\documentclass[11pt]{amsart}
\usepackage{amssymb,amsmath,amsthm, framed,graphicx,latexsym,epsfig}
\usepackage{stmaryrd}
\usepackage{enumerate, lineno}

\theoremstyle{plain} 
\numberwithin{equation}{section}
\newtheorem{theorem}{Theorem}[section]
\newtheorem{lemma}[theorem]{Lemma}
\newtheorem{corollary}[theorem]{Corollary}
\newtheorem{proposition}[theorem]{Proposition}

\theoremstyle{definition} 

\def\tr{\Delta}
\def\a{\alpha}
\def\bea{\begin{eqnarray*}}
\def\eea{\end{eqnarray*}}
\def\be{\begin{equation}}
\def\ee{\end{equation}}
\begin{document}

\title{Bach-flat $h$-Almost gradient Ricci solitons }

\author{Gabjin Yun}
\address{Department of Mathematics, Myong Ji University, San 38-2 Namdong, Yongjin, Gyeonggi, Korea}
\email{gabjin@mju.ac.kr}

\author{Jinseok Co}
\address{Department of Mathematics, Chung-Ang University, 84 Heukseok-ro, Dongjak-gu, Seoul, Korea }
\email{co1010@hanmail.net}

\author{Seungsu Hwang} 
\address{Department of Mathematics, Chung-Ang University, 84 Heukseok-ro, Dongjak-gu, Seoul, Korea }
\email{seungsu@cau.ac.kr}
\thanks{The first author was supported by the Basic Science Research Program
through the National Research Foundation of Korea (NRF) funded by the
Ministry of Education, Science and Technology (2011-0007465). The
third and corresponding author was supported by  the Ministry of Education (NRF-2015R1D1A1A01057661).}


\begin{abstract}
On an $n$-dimensional complete manifold $M$, consider an $h$-almost gradient Ricci soliton, which is a generalization of a gradient Ricci soliton. We prove that if the manifold is Bach-flat and $dh/du>0$, then the manifold $M$ is either Einstein or rigid. In particular, such a manifold has harmonic Weyl curvature. Moreover, if the dimension of $M$ is four, the metric $g$ is locally conformally flat.
\end{abstract}
 
\maketitle

\section{Introduction}
The notion of an $h$-almost Ricci soliton was introduced by Gomes, Wang, and Xia \cite{almost}. Such a soliton is a generalization of an almost Ricci soliton presented in \cite{Barros} and \cite{psrr}. 
An $h$-almost Ricci soliton is a complete Riemannian manifold $(M^n,g)$ with a vector field $X$ on $M$, a soliton function $\lambda :M\to{\mathbb R}$ and a signal function $h:M\to {\mathbb R}^+$ satisfying the equation
$$ r_g+\frac h2 \, {\mathcal L}_Xg=\lambda \, g,$$
where $r_g$ is the Ricci curvature of $g$. 
A function is called signal if it has only one sign; in other words, it is  either positive or negative on $M$. 
Let $(M,g,X,h,\lambda)$ denote an $h$-almost Ricci soliton.
In particular, $(M,g,\nabla u, h,\lambda)$ for some smooth function $u:M\to{\mathbb R}$ is called an $h$-almost gradient Ricci soliton with potential function $u$. In this case, we have
\be r_g+h\, D_gdu =\lambda \, g.\label{basic}\ee
Here, $D_gdu$ denotes the Hessian of $u$.
Note that if we take $u=e^{-\frac fm}$ and $h=-\frac mu$, 
then (\ref{basic}) becomes
$${\rm Ric}_f^m= r_g+D_gdf -\frac 1m df\otimes df =\lambda \, g.$$ 
In other words, the $(\lambda, n+m)$-Einstein equation is a special case of (\ref{basic}). Here, ${\rm Ric}_f^m$ is called the $m$-Bakry-Emery tensor.
For further details of $h$-almost Ricci solitons, we refer to \cite{almost}. 

In this paper we consider Bach-flat $h$-almost gradient Ricci solitons. The Bach tensor was introduced by R. Bach and this notion plays an important role in conformal relativity. On any $n$-dimensional Riemannian manifold $(M,g)$, $n\geq 4$, the Bach tensor is defined by 
$$ B= \frac 1{n-3} \, \delta^D \delta {\mathcal W}+\frac 1{n-2}\, \mathring{\mathcal W}z, $$
where ${\mathcal W}$ is the Weyl tensor, $z$ is the traceless Ricci tensor, and 
$\mathring{\mathcal W}z$ is defined by
$$ \mathring{\mathcal W}z (X,Y)=\sum_{i=1}^nz({\mathcal W}(X, E_i)Y, E_i)$$
for some orthonormal basis $\{E_i\}_{i=1}^n$. It is easy to see that if $(M,g)$ is either locally conformally flat or Einstein, then it is Bach-flat: $B=0$. When $n=4$, it is well  known that Bach-flat metrics on a compact manifold $M$ are critical points of the following functional 
$$ g\mapsto \int_M |{\mathcal W}|^2 dv_g.$$
It is clear that when $h=1$ and $\lambda $ is a positive constant, an $h$-almost gradient Ricci soliton reduces to a gradient shrinking Ricci soliton. Cao and Chen proved that a complete Bach-flat gradient shrinking Ricci soliton  is either Einstein or rigid \cite{bet2}. On the other hand, Qing and Yuan classified Bach-flat static spaces \cite{QY}.

Our main result is as follows, which can be considered as a generalization of \cite{bet2}.
\begin{theorem}\label{thm301}  
Let $(M, g, \nabla u, h, \lambda)$ be an $n$-dimensional Bach-flat $h$-almost gradient Ricci soliton with potential function $u$. Assume that each level set of $u$ is compact and $h$ is a function of $u$ only. 
Then, $(M,g, \nabla u, h, \lambda)$ is either
\begin{enumerate}
\item Einstein with constant functions $u$ and $h$, or
\item locally isometric to a warped product with $(n-1)$-dimensional Einstein fibers if $\frac {dh}{du}>0$ on $M$.
\end{enumerate}
\end{theorem}

For example, when $m>0$, $h=-\frac mu <0$ satisfies the condition of Theorem~\ref{thm301}, since
$$\frac{dh}{du}=\frac m{u^2}>0.$$
This recovers the result of \cite{bet1}. It will be interesting if one can weaken the condition of Theorem~\ref{thm301}.

In the case of (2) in Theorem~\ref{thm301}, a warped product metric has vanishing Cotten tensor (see (\ref{eqn31}) below) since its fiber is Einstein. Thus, 
as a consequence of Theorem~\ref{thm301}, we have the following.

\begin{corollary}\label{thm302}  
Let $(M, g, \nabla u, h, \lambda)$ be an $n$-dimensional Bach-flat $h$-almost gradient Ricci soliton with potential function $u$.
Assume that each level set of $u$ is compact and $h$ is a function of $u$ only. 
If $\frac {dh}{du}>0$ on $M$, then $(M,g)$ has harmonic Weyl curvature. 
\end{corollary}

In particular,  when $n=4$, the Einstein fibers in Theorem~\ref{thm301}  have constant curvature. A computation shows that such a metric is locally conformally flat, which proves the following theorem.
\begin{theorem}\label{thm303}
Let $(M, g, \nabla u, h, \lambda)$ be a $4$-dimensional Bach-flat $h$-almost gradient Ricci soltion with potential function $u$.  Assume that each level set of $u$ is compact and $h$ is a function of $u$ only with $\frac {dh}{du}>0$. Then, $(M,g)$ is locally conformally flat. 
\end{theorem}
We remark that, as in \cite{bet1}, Theorems~\ref{thm301}, Corollary~\ref{thm302}, and Theorem~\ref{thm303} can be extended to the case in which $M$ has a non-empty boundary. 

\section{Preliminaries}
In this section, we derive several useful identities containing various curvatures and the Cotton tensor.

We start with basic definitions of differential operators acting on tensors.
Let us denote by $C^{\infty}(S^2M)$ the space of sections of symmetric $2$-tensors on a Riemannian manifold $M$. 
Let $D$ be the Levi-Civita connection of $(M, g)$.
 Then the differential operator $d^D$ from $C^{\infty}(S^2M)$ into 
  $C^\infty\left(\Lambda^2 M \otimes T^*M\right)$ is defined as
$$ 
d^D \omega(X,Y,Z)= (D_X \omega)(Y,Z)-(D_Y \omega)(X,Z)
$$
for $\omega \in C^{\infty}(S^2M)$ and vectors $X, Y$, and $Z$. 
Let us denote by $\delta^D$ the formal adjoint operator of $d^D$.

 For a function $f \in C^\infty(M)$ and $\omega \in C^{\infty}(S^2M)$,
$df \wedge \omega $ is defined as
$$
(df \wedge \omega) (X,Y,Z)= df(X) \omega(Y,Z)-df(Y) \omega(X,Z).
$$
Here, $df$ denotes the usual total differential of $f$. We also denote by $\delta$  the negative divergence operator so that
$\Delta f = - \delta d f$.

Taking the trace of (\ref{basic}) gives
$$ s_g+h\, \tr u =n \, \lambda.$$
Thus, 
$$ ds_g +\tr u\, dh +h\, d\tr u =n\, d\lambda . $$
By taking the divergence of (\ref{basic}), we have
$$ -\frac 12 \, ds_g -D_gdu(\nabla h, \cdot) -h\, r_g(\nabla u, \cdot)-h\, d\tr u =-d\lambda. $$
By adding the previous two equations, we have
\be \frac 12 \, ds_g -D_gdu(\nabla h, \cdot) -h\, r_g(\nabla u, \cdot) +\tr u\, dh =(n-1)\, d\lambda. \label{eqn23}
\ee
Note that 
\be
\delta (hr_g(\nabla u, \cdot))=-r_g(\nabla u, \nabla h)-\frac h2 \langle \nabla s_g, \nabla u\rangle +|r_g|^2-\lambda s_g.
\label{eqn420}\ee
Therefore, we have the following equality.
\begin{proposition} On $M$ we have
\bea
(n-1)\tr \lambda &=& \frac 12\, \tr s_g +|r_g|^2-\lambda s_g-\frac h2 \langle \nabla s_g, \nabla u\rangle \\
& & +\left( \tr u -\frac {\lambda}h\right)\, \tr h   +\frac 1h \, \langle r_g, D_gdh\rangle  -2\, r_g(\nabla u, \nabla h).
\eea
\end{proposition}

On the other hand, by applying $d^D$ to (\ref{basic}), we have
\be d^Dr_g -\frac 1h\,  dh \wedge r_g +h\, \tilde{i}_{\nabla u} R = d\lambda \wedge g -\frac {\lambda }h \, dh \wedge g.\label{dr1}\ee
Here, an interior product $\tilde{i}$ of the final factor  is defined by 
$$\tilde{i}_{\xi}R(X,Y,Z)= R(X,Y, Z, \xi),$$ and we used the identity 
$$ d^DDdu =\tilde{i}_{\nabla u}R. $$

Hereafter, we denote $s_g$, $r_g$, and $D_gdu$ by $s$, $r$, and $Ddu$, respectively. 
From the curvature decomposition, we can compute that 
$$\tilde{i}_{\nabla u}R=\tilde{i}_{\nabla u}{\mathcal W}-\frac 1{n-2}i_{\nabla u} r\wedge g +\frac s{(n-1)(n-2)}du \wedge g -\frac 1{n-2}du \wedge r,$$
where $i_{\nabla u}r$ denotes the interior product defined by $$i_{\nabla u}r(X)=r(\nabla u, X).$$ 
The Cotton tensor $C$ is defined by
\be C=d^Dr -\frac 1{2(n-1)}ds\wedge g.\label{eqn31} \ee
Then, by (\ref{eqn23}) and (\ref{dr1}) as well as the fact that  $$s+h\, \tr u =n\lambda,$$
we have
\begin{eqnarray} C+h\,  \tilde{i}_{\nabla u}{\mathcal W}&=&  h\, D +\frac h{n-1}\, i_{\nabla u}r\wedge g +d\lambda\wedge g  -\frac 1{2(n-1)}\, ds \wedge g \nonumber \\
& &  +\frac 1h\, dh\wedge r -\frac {\lambda}h \, dh\wedge g \nonumber\\
&=&h\, D + H,\label{eqn24}
\end{eqnarray}
where $D$ is defined (as usual) by
\be (n-2)\, D=du\wedge r +\frac 1{n-1}i_{\nabla u}r \wedge g -\frac {s}{n-1}du \wedge g,\label{defndd}\ee
and $H$ is defined by
\bea H&=& -\frac 1{n-1}\, i_{\nabla h}Ddu\wedge g +dh\wedge \left( \frac 1h\, r +\frac {\tr u}{n-1}g-\frac {\lambda}h\, g\right)\\
&=& db \wedge r +\frac 1{n-1}\, i_{\nabla b}r \wedge g -\frac s{n-1}\, db\wedge g .
\eea
Here, $b=\log  |h|$ with $\nabla b =\frac {\nabla h}h$. In particular, $g^{ik}H_{ijk}=-g^{ik}H_{jik}=0$.

\begin{proposition} Let $(M,g, \nabla u, h, \lambda)$ be an $h$-almost gradient Ricci soliton with potential function $u$. Then
$$ C+h\, \tilde{i}_{\nabla u}{\mathcal W}= h\, D+H.$$
In particular, if $h$ is constant or $\frac {dh}{du}=0$, $H\equiv 0$.
\end{proposition}

\section{Bach-flat metrics}
In this section, we assume that $g$ is Bach-flat. 
Note that
$$\delta {\mathcal W}=-\frac {n-3}{n-2}\, C.$$
Recall that the Bach tensor is given by
$$ B =\frac 1{n-3}\, \delta^D\delta {\mathcal W}+\frac 1{n-2}\mathring{\mathcal W}z=\frac 1{n-2}\left( -\delta C +\mathring{\mathcal W}z\right).$$
Since
\bea
\delta(h\,  \tilde{i}_{\nabla u}{\mathcal W})(X,Y)&=&  -{\mathcal W}(\nabla h, X, Y, \nabla u)+h\, \delta {\mathcal W}(X,Y, \nabla u)\\
& & +h\, {\mathcal W}(X, E_i, Y, D_{E_i}du)\\
&=& {\mathcal W}(X, \nabla h, Y, \nabla u)- \frac {n-3}{n-2}\, h\, C(Y, \nabla u, X)- \mathring{\mathcal W}z,
\eea
by taking the divergence of (\ref{eqn24}) we have
\bea
-(n-2) B(X,Y)&=&-{\mathcal W}(X, \nabla h, Y, \nabla u)+\frac {n-3}{n-2}\, h\, C(Y, \nabla u, X) \\
& &-i_{\nabla h}D(X,Y)+h\, \delta D(X,Y)+\delta H(X,Y).
\eea
Hence, 
$$ -(n-2)B(\nabla u, \nabla u)= - D(\nabla h, \nabla u, \nabla u)+h\,  \delta D(\nabla u, \nabla u)+\delta H(\nabla u, \nabla u).
$$
As a result, from the assumption that  $B=0$ and $h$ is a function of $u$ only, 
$$ 0=\frac 1h D(\nabla h, \nabla u, \nabla u)= \delta D(\nabla u, \nabla u)+\frac 1h \, \delta H (\nabla u, \nabla u).$$
Let $\{E_i\}_{i=1}^n$ be a normal geodesic frame.
Note that,  since
$$ hD(E_i, D_{E_i}du, \nabla u)=-D(E_i, E_k, \nabla u)r_{ik}=0,$$
we have
$$\mbox{\rm div} (D(\cdot, \nabla u, \nabla u))=-\delta D(\nabla u, \nabla u)+ D(E_i, \nabla u, D_{E_i}du).$$
Furthermore,
\bea
|D|^2&=& \frac 1{n-2} \left( du(E_i)r(E_j,E_k)- du(E_j) r(E_i, E_k) \right)D_{ijk}\\
&=&  -\frac 2{n-2} \, D(E_i, \nabla u, E_k)\, r_{ik}\\
&=&	\frac {2\,h}{n-2}\,  D(E_i, \nabla u, D_{E_i}du).
\eea
Similarly, 
since
$$ hH(E_i, D_{E_i}du, \nabla u)= -  H(E_i, E_k, \nabla u)\, r_{ik}=0$$
and $h$ is a function of $u$ only,  we have
$$\mbox{\rm div} \left(\frac 1h\, H(\cdot, \nabla u, \nabla u)\right)=-\frac 1h \, \delta H(\nabla u, \nabla u)+ \frac 1h\, H(E_i, \nabla u, D_{E_i}du).$$
Moreover,
\bea |H|^2&=&  -\frac 2h\, H (E_i, \nabla h, E_k)\, r_{ik} = -\frac 2h\,  \frac {dh}{du}\, H(E_i, \nabla u, E_k)\, r_{ik}\\
&=& 2\,  \frac {dh}{du}\, H(E_i, \nabla u, D_{E_i}du).
\eea
Thus, 
\bea
0&=& \int_{t_1\leq u\leq t_2} \delta D(\nabla u, \nabla u)+\frac 1h \, \delta H (\nabla u, \nabla u)\\
&=& \frac {n-2}2\int_{t_1\leq u\leq t_2}\frac {|D|^2}h + \frac 12\, \int_{t_1\leq u\leq t_2} \frac {|H|^2}{h\,\frac {dh}{du}}.
\eea
Since $h$ is signal, $h$ is either positive or negative. For each case, we derive $D=H=0$ when $\frac {dh}{du}>0$. 
Therefore we have the following result. 
\begin{lemma} \label{lem22}
Let $(M, g, \nabla u, h, \lambda)$ be a Bach-flat $h$-almost gradient Ricci soliton with potential function $u$.
Assume that each level set of $u$ is compact and $h$ is a function of $u$ only. 
If $\frac {dh}{du}>0$ on $M$, then on $M$ we have
$$ D=H=0.$$
\end{lemma}

Now, since $D=H=0$, by (\ref{eqn31}) and (\ref{eqn24})
\be C= -h\, \tilde{i}_{\nabla u}{\mathcal W}.\label{eqn32}
\ee
By taking the divergence of (\ref{eqn32}), we have
$$ {\mathcal W}(X, \nabla h, Y, \nabla u)= \frac {n-3}{n-2}\, h\, C(Y, \nabla u, X).$$
By combining these equations, 
$$ \frac {n-3}{n-2}\, h^2\, C(Y, \nabla u, X)=-C(X, \nabla h, Y),$$
and 
$$ {\mathcal W}(X, \nabla h, Y, \nabla u)= -\frac {n-3}{n-2}\, h^2\, {\mathcal W}(X, \nabla u, Y, \nabla u).$$
Therefore, we have the following.
\begin{corollary}\label{cor33} When $D=H=0$, we have
\be {\mathcal W}(\cdot , \nabla u, \cdot, \nabla u)=C(\cdot, \nabla u, \cdot)=0,\label{eqnf22}\ee unless
$$ \frac {dh}{du}= -\left( \frac{n-3}{n-2}\right)\, h^2. $$
\end{corollary}
For example, when $h=-\frac mu$, (\ref{eqnf22}) holds if $m\neq 0$ or $-\frac {n-2}{n-3}$. Note that (\ref{eqnf22}) also holds if $h$ is constant.
\vskip .5pc
Moreover, we have the following result. 
\begin{lemma} \label{lem336} Suppose that $\frac {dh}{du}>0$. 
Then, for $X$ orthogonal to $\nabla u$, 
\be r(X, \nabla u)=0.\label{str1}\ee
In particular, $$i_{\nabla u}r=\a \, du,$$
where $\a =r(N, N)$ with $N=\nabla u/|\nabla u|$. 
\end{lemma}
\begin{proof}
By Lemma~\ref{lem22}, $D=H=0$.
From (\ref{dr1}), if $X$ is orthogonal to $\nabla u$,
$$
d^Dr(X,Y, \nabla u)= -\frac 1h\, dh(Y)\, r(X, \nabla u)+ d\lambda(X)du(Y).$$
Since $C (X,Y, \nabla u)=-h {\mathcal W}(X,Y, \nabla u, \nabla u)=0$ by (\ref{eqn32}), by (\ref{eqn31})  we have
$$  d^Dr(X, Y, \nabla u)= \frac 1{2(n-1)}\, ds(X)\, du(Y).
$$
Thus, by (\ref{eqn23})
\bea
\frac 1h\, \frac {dh}{du}\, r(X, \nabla u)&=& d\lambda(X)-\frac 1{2(n-1)}\, ds(X)\\
&=& \frac 1{(n-1)h}\, \left(\frac {dh}{du}- h^2\right)\, r(X,\nabla u),
\eea
which implies that 
$$\left( (n-2)\frac {dh}{du}+h^2\right) \, r(X, \nabla u)=0.$$
This completes the proof of our lemma.
\end{proof}

Note that Lemma~\ref{lem336}  holds with the assumptions that $D=H=0$ and 
\be \frac {dh}{du}\neq -\frac 1{n-2}\, h^2 \label{eqn209}\ee
 without $\frac {dh}{du}>0$. For example,  in the case of $m$-Bakry-Emery tensor, $h=-\frac mu$ satisfies (\ref{eqn209}) if $m\neq 2-n$.

\section{Level sets of $u$}
In this section, we will investigate the structure of  regular level sets of the potential function $u$. 
For a regular value $c$, we denote the level set $u^{-1}(u)$ by $L_c$. On $L_c$, 
let $\{E_i\}$, $1\leq i\leq n$, be an orthonoromal frame with  $E_n=N=\nabla u/|\nabla u|$. 

Futhermore, throughout the section we assume that $D=H=0$ with 
$$\frac {dh}{du}\neq -\left( \frac {n-3}{n-2}\right)\, h^2 \quad \mbox{and}\quad 
 \frac {dh}{du}\neq-\frac 1{n-2}\, h^2.$$
 Then, by Corollary~\ref{cor33}, (\ref{eqnf22}) and (\ref{str1}) hold. 
 Furthermore, for $X$ orthogonal to $\nabla u$, by the proof of Lemma~\ref{lem336},
$$d\lambda (X)=\frac 1{2(n-1)}\, ds (X).$$
Thus,
$  s+2(1-n)\lambda $ is constant on each level set of $u$.
Furthermore,
$$ \frac 12 \, X(|\nabla u|^2)= \langle D_Xdu, \nabla u\rangle =\frac 1h\, \left( \lambda du(X)- r(X, \nabla u)\right)=0,$$
which implies that $|\nabla u|^2$ is constant on each level set of $u$.  Therefore, we have the following.
\begin{lemma}\label{lem540}
$|\nabla u|^2$ and $s +2(1-n)\, \lambda$ are constant on each regular level set of $u$.
\end{lemma}

For further investigation, we need the following key lemma.
\begin{lemma} \label{lem459}
\bea
0&=&\frac {ns-(n-1)^2\lambda-\a}{(n-1)h}\, r-D_{\nabla u} r -\frac {r\circ r}h +\frac {n-3}{2(n-1)}\,  du\otimes ds\\
&+ &\frac 1{n-1}\left( ds(u)-\langle \nabla u, \nabla \a\rangle\right) \, g
  +\frac {s+(1-n)\lambda }{(n-1)h}(\a -s)g +\frac 1{n-1}du \otimes d\a.
\eea
\end{lemma}
\begin{proof}  To find $\delta D$, by (\ref{defndd}), we first compute
$$ \delta (du\wedge r)= \frac {s-(n-1)\lambda}h\, r-D_{\nabla u} r -\frac {r\circ r}h +\frac 12 du\otimes ds.$$
By Lemma~\ref{lem336}, $i_{\nabla u}r=\a\, du$. Thus,
$$ \delta (i_{\nabla u}r\wedge g)= -\langle \nabla u, \nabla \a\rangle g +\frac {s+(1-n)\lambda }h\a g +du \otimes d\a -\frac {\a}h\, r.
$$
Similarly,
$$ -\delta (s\, du\wedge g)=ds(u)\, g -\frac {s^2+(1-n)s\lambda }h\, g -du\otimes ds +\frac sh\, r.
$$
Hence, by (\ref{defndd}) together with  (\ref{str1}), we have
\bea
(n-2)\, \delta D&=&\frac {ns-(n-1)^2\lambda-\a}{(n-1)h}\, r-D_{\nabla u} r -\frac {r\circ r}h \\
& & +\frac {n-3}{2(n-1)}\,  du\otimes ds +\frac 1{n-1}\, du \otimes d\a\\
& &   +\frac 1{n-1}\left( ds(u)- \langle \nabla u, \nabla \a\rangle + \frac {s+(1-n)\lambda }{h}(\a -s) \right)g.
\eea
Since $D=\delta D=0$, the proof follows.
\end{proof}
Thus, we have the following.
\begin{corollary}  $(n-3)\, s+2\a$ is constant on each regular level set of $u$.\label{cor541}
\end{corollary}
\begin{proof}
Let $X$ be a vector orthogonal to $\nabla u$. By putting $(X, \nabla u)$ in the equation in Lemma~\ref{lem459}, 
\be D_{\nabla u}r(X, \nabla u)=0.\label{eqn506}\ee
Now, by putting $(\nabla u, X)$ in the equation in Lemma~\ref{lem459} again, we have
$$ 0=\frac {n-3}{2(n-1)}\, |\nabla u|^2 \, ds(X)+ \frac 1{n-1}\, |\nabla u|^2 \, d\a (X),$$
since $r(X, \nabla u)=0$ and $$D_{\nabla u}r(\nabla u, X)=D_{\nabla u}r(X, \nabla u).$$
\end{proof}

\begin{lemma} $s_g+2(1-n)\a$ is constant on each regular level set of $u$. \label{lem542}
\end{lemma}
\begin{proof} For $X$ orthogonal to $\nabla u$, by (\ref{eqnf22}) and (\ref{eqn506})
\bea
0&=& C(X, \nabla u, \nabla u)\\
&=& D_Xr(\nabla u, \nabla u)-\frac 1{2(n-1)}\, ds(X)\, |\nabla u|^2.
\eea
Thus, 
\bea
X(\a)&=& \frac 1{|\nabla u|^2}\, X\left(r(\nabla u, \nabla u)\right)\\
&=&  \frac 1{|\nabla u|^2}\, \left( D_Xr(\nabla u, \nabla u) +2\, r(D_Xdu, \nabla u)  \right)\\
&=& \frac 1{2(n-1)}\, ds(X),
\eea
since
$$ r(D_Xdu, \nabla u)= \frac 1h\, (\lambda\, r(X, \nabla u)-r\circ r(X, \nabla u))=0.$$
\end{proof}

By combining Lemma~\ref{lem540}, Corollary~\ref{cor541}, and Lemma~\ref{lem542}, we have the following.
\begin{theorem} Let $(M,g,\nabla u, h, \lambda)$ be a Bach-flat $h$-almost gradient Ricci soliton with potential function $u$. Assume that each level set of $u$ is compact and $h$ is a function of $u$ only with $\frac{dh}{du}>0$.
Then
$s_g$, $\a$, and $\lambda$ are constant on each regular level set of $u$.
In particular, if $h$ is constant, the condition on $\frac {dh}{du}$ is not necessary.
 \label{thm304}
\end{theorem} 

When $D=0$, the Ricci tensor has the following characterization.

\begin{lemma}\label{lem1-1} Suppose that $D=0$.  Then
the Ricci curvature tensor has at most two eigenvalues.
\end{lemma}
\begin{proof}
Let $\{E_i\}$, $1\leq i\leq n$, be an orthonoromal frame with  $E_n=N=\nabla u/|\nabla u|$. Then 
\be II_{ij}= \frac 1{h\, |\nabla u|}\left( \lambda g_{ij}-r_{ij}\right),\label{ast1}\ee
and
$$ m=\mbox{\rm tr}\, II = \frac {n-1}{h\, |\nabla u|}\left( \lambda + \frac {\a-s}{n-1}\right).$$
Thus, $m$ is constant on each level set of $u$, and
\bea
|II- \frac m{n-1} \,g|^2&=& |II|^2- \frac {m^2}{n-1}\\
&=& \frac 1{h^2|\nabla u|^2}\, \left( |r|^2-\a^2 -\frac {(s-\a)^2}{n-1}\right)\\
&=&\frac 1{h^2|\nabla u|^2}\, \left( |r|^2-\frac n{n-1}\, \a^2+\frac {2s\a}{n-1} -\frac {s^2}{n-1}\right).
\eea
Since $r\circ r(\nabla u, \nabla u)=\a^2 |\nabla u|^2$, from the identity 
$$ \frac {n-2}2\, |D|^2= |r|^2|\nabla u|^2-\frac n{n-1}\, r\circ r(\nabla u, \nabla u) +\frac {2s}{n-1}\, r(\nabla u, \nabla u)-\frac {s^2}{n-1}|\nabla u|^2,$$
we have
$$ |D|^2=\frac 2{n-2}\, h^2\, |\nabla u|^4\, |II-\frac m{n-1}\, g |^2.
$$
Since $D=0$, 
we have 
\be II_{ij}=\frac m{n-1}g_{ij},\label{mean1208}\ee
which implies that 
\be r_{ij}= \frac {s-\a}{n-1}\, g_{ij}\label{subric}\ee
for $i=1,...,n-1$ by (\ref{ast1}). 
This completes the proof of our lemma.
\end{proof}
As an immediate consequence, on an open set  $\{x\in M\, \vert\, \nabla u  (x)\neq 0\}$, 
the Ricci tensor may be written as
$$ r_g=\beta \, du\otimes du + \left( \frac {s-\a}{n-1}\right)\, g,
$$
where 
$$\beta = \frac {n\,\a -s}{(n-1)|\nabla u|^2}.$$
Thus, by (\ref{basic}) we have
$$ D_gdu = \frac 1h\left( \lambda +\frac {\a-s}{n-1}\right)\, g -\frac {\beta}h\, du \otimes du.
$$
\vskip .5pc
Now, we are ready to prove Corollary~\ref{thm302}, which shows the relationship between Bach-flat metrics and harmonic Weyl metrics.
\vskip 1.0pc\noindent
{\it Proof of Corollary~\ref{thm302}}.\, 
Note that, by (\ref{eqn32}) and (\ref{eqnf22})
$$ C(\cdot, \cdot, \nabla u)=C(\cdot,  \nabla u, \cdot)=0.$$
On the other hand, by the Codazzi equation,
$$ \langle R(X,Y)Z, N\rangle=D_YII(X,Z)-D_XII(Y,Z).$$
Thus, for $1\leq i,j,k\leq n-1$,  by (\ref{mean1208})
\bea \langle R(E_i, E_j)E_k, N\rangle &=& E_j(II(E_i, E_k))-II(D_{E_j}E_i, E_k)-II(E_i, D_{E_j}E_k)\\
& -&E_i(II(E_j, E_k))+II(D_{E_i}E_j, E_k)+II(E_j, D_{E_i}E_k)=0.
\eea
Therefore, by (\ref{dr1}) 
$$ d^Dr(E_i, E_j, E_k)=0,$$
which implies that 
$$ C(E_i, E_j, E_k)=d^Dr(E_i, E_j,E_k)-\frac 1{2(n-1)}\, ds\wedge g(E_i, E_j, E_k)=0.$$
Hence, $C$ is identically zero, and so is 
$ \delta {\mathcal W}$.
\hfill $\Box$
\vskip 1.0pc 

The following is a restatement of Theorem~\ref{thm301}.

\begin{theorem}
Let $(M,g, \nabla u, h, \lambda)$ be a Bach-flat $h$-almost gradient Ricci soliton with potential function $u$.
Assume that each level set of $u$ is compact with $\frac {dh}{du}>0$ on $M$. Then, either  $g$ is  Einstein with constant function $u$ or  the metric can be written as
$$ g=dt^2+\psi^2(t)\,  \hat{g}_E,$$
where $\hat{g}_E $ is the Einstein metric on the level set $E=L_{c_0}$ for some $c_0$. \label{thm901}
\end{theorem}
\begin{proof} Assume that $u$ is not constant.
By Lemma~\ref{lem22}, $D=H=0$.
Since $|\nabla u|^2$ depends only on $u$ by Lemma~\ref{lem1-1},
as shown in the proof of Theorem 7.9 of \cite{hpw} with Remark 3.2 of \cite{bet2}, 
the metric can be locally written as 
$$ g=dt^2+ \hat{g}_{c}.$$ 
Here, $\hat{g}_c$ denotes the induced metric on the level set $L_c=u^{-1}(c)$ for each regular value $c$. Furthermore, $(L_c, \hat{g}_c)$ is necessarily Einstein;
by the Gauss equation
$$ \hat{R}_{ijij}=R_{ijij}+II_{ii}II_{jj}-II_{ij}^2 = R_{ijij}+\frac {m^2}{(n-1)^2}.$$
Thus, 
$$ \hat{r}_{ii}= r_{ii} -R(N, E_i, N, E_i)+\frac {m^2}{n-1}.$$
By (\ref{eqnf22}) and (\ref{subric}), we have
$$ R(E_i, N, E_i, N)=  \frac 1{n-2}\, (r_{ii}+\a)-\frac s{(n-1)(n-2)}=\frac {\a}{n-1}.
$$
Hence, it follows that
$$ \hat{r}_{ii}= r_{ii}+\frac {m^2-\a}{n-1}=\frac 1{n-1}\, (s-2\alpha +m^2)=\hat{\lambda}_0.
$$
Since $s$, $\alpha$, and  $m$ are constant along $L_c$, this proves that  $(L_c, \hat{g}_c)$ has constant Ricci curvature. 
As a result, by suitable change of variable, the metric $g$ can be written as in the statement of Theorem~\ref{thm901}.
\end{proof}
\vskip 1.2pc\noindent
{\bf  Acknowledgment}
The authors would like to express their gratitude to the referee for several valuable comments.


\begin{thebibliography}{99}
\bibitem{Be} Besse, A.L., Einstein Manifolds. Ergeb. Math. Grenzgeb.
 (3) 10, A Series of Modern Surveys in Mathematics, Springer-Verlag, Berlin, 1987.
 \bibitem{Barros}  Barros, A., Ribeiro Jr., E. Some characterizations for compact almost Ricci solitons. {\it Proc. Amer.
Math. Soc.} {\bf 140} (2012), 1033--1040.
 \bibitem{ccc}
 Cao, H. D., Catino, G., Chen, Q., Mantegazza, C., Mazzieri, L., Bach-flat gradient steady Ricci solitons, {\it Calc. Var. Partial Differential Equations} {\bf 49} (2014), no. 1-2, 125--138.
\bibitem{bet2} Cao, H. D., Chen, Q. On Bach-flat gradient shrinking Ricci solitons, {\it Duke Math. J.} {\bf 162} (2013), no. 6, 1149--1169. 
\bibitem{bet1} Chen, Q., He C., On Bach-flat warped product Einstein manifolds, {\it Pacific J. Math.} {\bf 265} (2013), no. 2, 313--326.
\bibitem{almost} Gomes, J. N., Wang, Q., Xia, C., On the $h$-almost Ricci soliton, arXiv:1411.6416v2 [math.DG], 2014.
\bibitem{hpw} He, C., Petersen, P., Wylie, W., {On the classification of warped product Einstein metrics}, {\it Comm. Anal. Geom.} {\bf 20} (2012), no. 2, 271--311.
\bibitem{psrr} Pigola, S., Rigoli, M., Rimoldi, M., Setti, A. G., Ricci almost solitons. {\it Ann. Sc. Norm. Super. Pisa Cl. Sci. (5)}  {\bf 10}  (2011),  no. 4, 757--799.
\bibitem{QY} Qing, J., Yuan W., A note on static spaces and related problems, {\it J. of Geom. and Phys. }{\bf 74} (2013), 18--27.
\end{thebibliography}
\end{document}